\documentclass[11pt]{article}
\usepackage{amsmath,amssymb}

\newtheorem{propo}{{\bf Proposition}}[section]
\newtheorem{coro}[propo]{{\bf Corollary}}
\newtheorem{lemma}[propo]{{\bf Lemma}} \newtheorem{theor}[propo]{{\bf
Theorem}} \newtheorem{ex}{{\sc Example}}[section]

\newenvironment{proof}{{\bf Proof.}}{$\Box$}

\def\N{{\mathbb N}}

\begin{document}

\vspace*{1.0in}

\begin{center}THE GENERALISED NILRADICAL OF A LIE ALGEBRA
\end{center}
\bigskip

\begin{center} DAVID A. TOWERS 
\end{center}
\bigskip

\begin{center} Department of Mathematics and Statistics

Lancaster University

Lancaster LA1 4YF

England

d.towers@lancaster.ac.uk 
\end{center}
\bigskip

\begin{abstract}
A solvable Lie algebra $L$ has the property that its nilradical $N$ contains its own centraliser. This is interesting because gives a representation of $L$ as a subalgebra of the derivation algebra of its nilradical with kernel equal to the centre of $N$. Here we consider several possible generalisations of the nilradical for which this property holds in any Lie algebra. Our main result states that for every Lie algebra $L$, $L/Z(N)$, where $Z(N)$ is the centre of the nilradical of $L$, is isomorphic to Der$(N^*)$ where $N^*$ is an ideal of $L$ such that $N^*/N$ is the socle of a semisimple Lie algebra.

\end{abstract}

\section{Introduction}
Throughout, $L$ will be a finite-dimensional Lie algebra, over a field $F$, with nilradical $N$ and radical $R$. If $L$ is solvable, then $N$ has the property that $C_L(N) \subseteq N$. This property supplies a representation of $L$ as a subalgebra of Der($N$) with kernel $Z(N)$. The purpose of this paper is to seek a larger ideal for which this property holds in all Lie algebras. The corresponding problem has been considered for groups (see, for example, Aschbacher \cite[Chapter 11]{asch}).  In group theory, the quasi-nilpotent radical (also called by some the generalised Fitting subgroup), $F^*(G)$, of a group $G$ is defined to be $F(G)+E(G)$, where $F(G)$ is the Fitting subgroup and $E(G)$ is the set of {\em components} of $G$: that is,  the quasi-simple subnormal subgroups of the group. It is also equal to the socle of $C_G(F(G))F(G)/F(G)$. The generalised Fitting subgroup, $\tilde{F}(G)$, is defined to be the socle of $G/\Phi(G)$, where $\Phi(G)$ is the Frattini subgroup of $G$ (see, for example, \cite{mv}).  Here we consider various possible analogues for Lie algebras.
\par

First we introduce some notation that will be used. The {\em centre} of $L$ is $Z(L)=\{ x\in L : [x,y]=0$ for all $y \in L\}$; if $S$ is a subalgebra of $L$, the {\em centraliser} of $S$ in $L$ is $C_L(S)=\{x\in L:[x,S] =0\}$; the {\em Frattini ideal}, $\phi(L)$, of $L$ is the largest ideal contained in all of the maximal subalgebras of $L$; we say that $L$ is $\phi$-free if $\phi(L)=0$; the {\em socle} of $S$, Soc\,$S$, is the sum of all of the minimal ideals of $S$; and the {\em $L$-socle} of $S$, Soc$_L\,S$, is the sum of all of the minimal ideals of $L$ contained in $S$. The symbol `$\oplus$' will be used to denote an algebra direct sum, whereas `$\dot{+}$' will denote a direct sum of the vector space structure alone.
\par

We call $L$ {\em quasi-simple} if $L^2 = L$ and $L/Z(L)$ is simple. Of course, over a field of characteristic zero a quasi-simple Lie algebra is simple, but that is not the case over fields of prime characteristic. For example, $A_n$ where $n\equiv -1(modp)$ is quasi-simple, but not simple. This suggests using the quasi-simple subideals of a Lie algebra $L$ to define a corresponding $E(L)$. However, first note that quasi-simple subideals of $L$ are ideals of $L$. This follows from the following easy lemma.

\begin{lemma}\label{l:char} If $I$ is a perfect subideal (that is, $I^2=I$) of $L$ then $I$ is a characteristic ideal of $L$.
\end{lemma}
\begin{proof} If $I$ is perfect then $I=I^n$ for all $n \in \N$. It follows that $[L,I]=[L,I^n] \subseteq L$ (ad $I)^n \subseteq I$ for some $n \in \N$, and hence that $I$ is an ideal of $L$.  But now, if $D \in$ Der$(L)$, then $D([x_1,x_2])=[x_1,D(x_2)]+[D(x_1),x_2] \in I$ for all $x_1, x_2 \in I$. Hence $D(I)=D(I^2)\subseteq I$.
\end{proof}
\bigskip

Combining this with the preceding remark we have the following.

\begin{lemma}\label{l:char0} Let $L$ be a Lie algebra over a field of characteristic zero. Then $I$ is a quasi-simple subideal of $L$ if and only if it is a simple ideal of $L$.
\end{lemma}
\medskip

We say that an ideal $A$ of $L$ is {\em quasi-minimal} in $L$ if $A/Z(A)$ is a minimal ideal of $L/Z(A)$ and $A^2=A$. Clearly a quasi-simple ideal is quasi-minimal. Over a field of characteristic zero, an ideal $A$ of $L$ is quasi-minimal if and only if it is simple. So an alternative is to define $E(L)$ to consist of the quasi-minimal ideals of $L$. We investigate these two possibilities in sections $3$ and $5$. 
\par

In sections $4$ and $6$ our attention turns to two further candidates for a generalised nilradical: the $L$-socle of $(N+C_L(N))/N$ and the socle of $L/\phi(L)$. All of these possibilities turn out to be related, but not always equal.

\section{Preliminary results}
Let $L$ be a Lie algebra over a field $F$ and let $U$ be a subalgebra of $L$. If $F$ has characteristic $p>0$ we call $U$ {\em nilregular} if the nilradical of $U$, $N(U)$, has nilpotency class less than $p-1$. If $F$ has characteristic zero we regard every subalgebra of $L$ as being nilregular. We say that $U$ is {\em characteristic in $L$} if it is invariant under all derivations of $L$. Then we have the following result.

\begin{theor}\label{t:nil} \begin{itemize}
\item[(i)] If $I$ is a nilregular ideal of $L$ then $N(I) \subseteq N(L)$.
\item[(ii)] If $I$ is a nilregular subideal of $L$ and every subideal of $L$ containing $I$ is nilregular, then $N(I) \subseteq N(L)$.
\end{itemize}
\end{theor}
\begin{proof}\begin{itemize}
\item[(i)] We have that $N(I)$ is characteristic in $I$. This is well-known in characteristic zero, and is given by \cite[Corollary 1]{mak} in characteristic $p$. Hence it is a nilpotent ideal of $L$ and the result follows.
\item[(ii)] Let $I=I_0 < I_1 < \ldots < I_n=L$ be a chain of subalgebras of $L$ with $I_j$ an ideal of $I_{j+1}$ for $j=0, \ldots, n-1$. Then $N(I) \subseteq N(I_1) \subseteq .... \subseteq N(I_n)=N(L)$, by (i).
\end{itemize}
\end{proof}
\medskip

Similarly, we will call the subalgebra $U$ {\em solregular} if the underlying field $F$ has characteristic zero, or if it has characteristic $p$ and the (solvable) radical of $U$, $R(U)$, has derived length less than $log_2 p$. Then we have the following corresponding result.

\begin{theor}\label{t:solv}  \begin{itemize}
\item[(i)] If $I$ is a solregular ideal of $L$ then $R(I) \subseteq R(L)$.
\item[(ii)] If $I$ is a solregular subideal of $L$ and every subideal of $L$ containing $I$ is solregular, then $R(I) \subseteq R(L)$.
\end{itemize}
\end{theor}
\begin{proof}  This is similar to the proof of Theorem \ref{t:nil}, using \cite[Theorem 2]{pet}.
\end{proof}
\bigskip

We also have the following result which we will improve upon below, but by using a deeper result than is required here.

\begin{theor}\label{t:minp} Let $L$ be a Lie algebra over a field $F$, and let $I$ be a minimal non-abelian ideal of $L$. Then either 
\begin{itemize}
\item[(i)] $I$ is simple or 
\item[(ii)] $F$ has characteristic $p$, $N(I)$  has nilpotency class greater than or equal to $p-1$, and $R(I)$ has derived length greater than  or equal to $log_2 p$.
\end{itemize}
\end{theor}
\begin{proof} Let $I$ be a non-abelian minimal ideal of $L$ and let $J$ be a minimal ideal of $I$. Then $J^2=J$ or $J^2=0$. The former implies that $J$ is an ideal of $L$ by Lemma \ref{l:char}, and hence that $I$ is simple.  So suppose that $J^2=0$. Then $N(I) \neq 0$ and $R(I) \neq 0$. But if $I$ is nilregular we have that $N(I) \subseteq N(L)\cap I=0$, since $I$ is non-abelian, a contradiction. Similarly, if $I$ is solvregular, then $R(I) \subseteq R(L)\cap I=0$, a contradiction. The result follows.
\end{proof}
\bigskip

As a result of the above we will call the subalgebra $U$ {\em regular} if it is either nilregular or solregular; otherwise we say that it is {\em irregular}. 
Then we have the following corollary.

\begin{coro}\label{c:minp} Let $L$ be a Lie algebra over a field $F$. Then every minimal ideal of $L$ is abelian, simple or irregular.
\end{coro}

Block's Theorem on differentiably simple rings (see \cite{block}) describes the irregular minimal ideals as follows.

\begin{theor}\label{t:block} Let $L$ be a Lie algebra over a field of characteristic $p>0$ and let $I$ be an irregular minimal ideal of $L$. Then $I \cong  S\otimes {\mathcal O}_n$, where $S$ is simple and ${\mathcal O}_n$ is the truncated polynomial algebra in $n$ indeterminates. Moreover, $N(I)$ has nilpotency class $p-1$ and $R(I)$ has derived length $\lceil log_2 p \rceil$.
\end{theor}
\begin{proof} Every non-abelian minimal ideal $I$ of $L$ is ad$\mid_I(L)$-simple, so the first assertion follows from \cite[Theorem 1]{block}. Now $N(I)=R(I)\cong S\otimes {\mathcal O}_n^+$, where ${\mathcal O}_n^+$ is the augmentation ideal of ${\mathcal O}_n$. It is then straightforward to check that the final assertion holds. 
\end{proof}
\bigskip

Note that if $\mathcal{N}$ and $\mathcal{S}$ are the classes of Lie algebras that are themselves nilregular and solregular respectively, then $\mathcal{N} \not \subseteq \mathcal{S}$ and $\mathcal{S} \not \subseteq \mathcal{N}$, as the following examples show. 

\begin{ex}\label{e:fil} Let $L$ be a filiform nilpotent Lie algebra of dimension $n$ over a field $F$. Then $L$ has nilpotency class $n-1$ and derived length $2$. Thus, if $F$ has characteristic $p>3$, and $n \geq p$, then $L$ has nilpotency class greater than or equal to $p-1$, and so is not nilregular. However, it is solregular, since $2<log_2p$.
\end{ex}

\begin{ex}\label{e:sol} Let $L=Fe_1+Fe_2$ with product $[e_1,e_2]=e_2$ and let $F$ have characteristic $3$. The $N(L)=Fe_2$ has nilpotency class $1<p-1$ and so $L$ is nilregular. But $R(L)=L$, so $L$ has derived length $2>log_2p$ and is not solregular.
\end{ex}
 For every Lie algebra $L^{(n)}\subseteq L^{2^n}$, so any nilregular nilpotent Lie algebra of nilpotency class $2^n$ is solregular, since $2^n<p-1<p$ implies that $n<log_2p$. However, it is not true generally that a nilpotent nilregular Lie algebra is solregular, as the following example shows.

\begin{ex}\label{e:nilp} Let $L$ be the seven-dimensional Lie algebra over a field $F$ of characteristic $p=7$ with basis $e_1, \ldots, e_7$ and products $[e_2,e_1]=e_4$, $[e_3,e_1]=e_5$, $[e_3,e_2]=e_5$, $[e_4,e_3]=-e_6$, $[e_5,e_1]=e_7$, $[e_5,e_2]=2e_6$, $[e_5,e_4]=e_7$, $[e_6,e_1]=e_7$ and $[e_6,e_2]=e_7$ (see \cite[page 87]{bok}). Then $L$ has nilpotency class $5<p-1$ and so is nilregular, but its derived length is $3>log_2p$, so it is not solregular.
\end{ex}

We also have the following result.

\begin{coro}\label{c:equiv} If $L$ is a Lie algebra and $A$ is a regular ideal of $L$, then $A$ is quasi-minimal in $L$ if and only if it is a quasi-simple ideal of $L$.
\end{coro}

However, the above result is not true for all ideals, as the following example shows. 

\begin{ex}\label{e:pasha} Let $L=sl(2)\otimes {\mathcal O}_m+1\otimes {\mathcal D}$, where ${\mathcal O}_m$ is the truncated polynomial algebra in $m$ indeterminates, ${\mathcal D}$ is a non-zero solvable subalgebra of Der(${\mathcal O}_m$), ${\mathcal O}_m$ has no ${\mathcal D}$-invariant ideals, and the ground field is algebraically closed of characteristic $p>5$. Then $L$ is semisimple and  $A=sl(2)\otimes {\mathcal O}_m$ is the unique minimal ideal of $L$ (see \cite[Theorem 6.4]{pasha}). Since $Z(A)=0$, $A$ is clearly quasi-minimal but not quasi-simple.
\end{ex}

If $S$ is a subalgebra of $L$, we denote by $R_c(S)$ the (solvable) characteristic radical of $S$; that is, the sum of all of the solvable characteristic ideals of $L$. (see Seligman \cite{sel}).

\begin{theor}\label{t:rad} Let $L$ be a Lie algebra over any field $F$. Then $R_c(C_L(N))=Z(N)$. Moreover, if $C_L(N)$ is regular, then $R_c(C_L(N))=R(C_L(N))$. 
\end{theor}
\begin{proof} Let $Z=Z(N)$, $\overline{L}=L/Z$ and $H=R_c(C_L(N))$. Then $H$ is a characteristic ideal of $C_L(N)$, and hence an ideal of $L$. Assume that $\overline{H}\neq 0$. Then there exists $k \geq1$ such that $H^{(k+1)} \subseteq Z$ but $X=H^{(k)} \not \subseteq Z$. Then $X^2 \subseteq Z$ and $X^3 \subseteq [N,C_L(N)]=0$, since $X \subseteq C_L(N)$. It follows that $X$ is a nilpotent ideal of $L$, and hence that $X \subseteq N$. But $[X,N]=0$, giving $X\subseteq Z$, a contradiction.
\par

Now suppose that $C_L(N)$ is nilregular. Then, clearly, $R_c(C_L(N)) \subseteq R(C_L(N))$. Suppose that $R(C_L(N)) \neq Z$. Let $A/Z$ be a minimal ideal of $C_L(N)/Z$ with $A\subseteq R(C_L(N))$. Then $A^3=0$ and so $A\subseteq N(C_L(N)) \subseteq N(L)$, by Theorem \ref{t:nil} (i). Hence $A=Z$, a contradiction.  
\par

Finally, suppose that $C_L(N)$ is solregular. Then $R(C_L(N))=R(L)\cap C_L(N)$ is an ideal of $L$, and arguing as in the first paragraph of this proof shows that $R(C_L(N))=Z(N)$.  
\end{proof}
\bigskip

This has the following useful corollary.

\begin{coro}\label{c:cent0} Let $L$ be a Lie algebra over a field $F$, let $N$ be its nilradical and let $C=C_L(N)$ be regular. Then 
\begin{itemize}
\item[(i)] if $\phi(C)\cap Z(N)=0$, $C=Z(N)\dot{+}B$ where $B$ is a semisimple subalgebra of $L$ and $B^2$ is an ideal of $L$;
\item[(ii)] if $\phi(L)\cap Z(N)=0$, $C=Z(N)\oplus B$ where $B$ is a maximal semisimple ideal of $L$; and
\item[(iii)] if $F$ has characteristic zero, then  $C = Z(N) \oplus S$ where $S$ is the maximal semisimple ideal of $L$. 
\end{itemize}
\end{coro}
\begin{proof} \begin{itemize}
\item[(i)] Suppose that $\phi(C)\cap Z(N)=0$. Then $C=Z(N)\dot{+} B$ for some subalgebra $B$ of $C$, by \cite[Lemma 7.2]{frat}. Moreover, $B\cong C/Z(N)$ is semisimple, by Theorem \ref{t:rad}, and $B^2=C^2$ is an ideal of $L$.
\item[(ii)] Suppose that $\phi(L)\cap Z(N)=0$. The $L=Z(N)\dot{+} U$ for some subalgebra $U$ of $L$, by \cite[Lemma 7.2]{frat} again. It follows that $C=Z(N)\oplus  B$ where $B=C\cap U$, which is an ideal of $L$, and $B$ is semisimple. Moreover, if $S$ is a semisimple ideal of $L$ with $B\subseteq S$, then $[S,N]\subseteq S\cap N=0$, so $S\subseteq C$. Hence $S=B$.
\item[(iii)] So suppose now that $F$ has characteristic zero. Then $C=Z(N)\dot{+} B$ where $B$ is a Levi factor of $C$. Also, $B=B^2=C^2$ is an ideal of $L$, so $C=Z(N)\oplus B$. Moreover, if $S$ is the maximal semisimple ideal of $L$, then $B \subseteq S$ and $[S,N]\subseteq S\cap N=0$, so $S \subseteq C$. It follows that $S=B$. 
\end{itemize}
\end{proof}
\bigskip

Finally, the following straightforward results will prove useful.

\begin{lemma}\label{l:centre} Let $K$ be an ideal of $L$ with $K \subseteq C_L(N)$. Then $Z(K)=Z(N) \cap K$.
\end{lemma}
\begin{proof} Clearly $Z(K)$ is an abelian ideal of $L$, so $Z(K) \subseteq N$. Moreover, $[Z(K),N] \subseteq [K,N] =0$, so $Z(K) \subseteq Z(N) \cap K$. Also $[Z(N) \cap K,K] \subseteq [N,K]=0$, so $Z(N) \cap K \subseteq Z(K)$.
\end{proof}

\begin{lemma}\label{l:phi} Let $L$ be any Lie algebra and suppose that $A$ is an ideal of $L$ with $A^2=A$. Then $Z(A) \subseteq \phi(L)$. If $A$ is a quasi-minimal ideal of $L$, then $Z(A)=A\cap \phi(L)$.
\end{lemma}
\begin{proof} Suppose that $Z(A) \not \subseteq \phi(L)$. Then there is a maximal subalgebra $U$ of $L$ such that $L=Z(A)+ U$. Thus $A=Z(A)+ U\cap A$ and $U\cap A$ is an ideal of $L$. It follows that $A=A^2=(U\cap A)^2\subseteq U\cap A \subseteq A$, whence $Z(A)\subseteq U$, a contradiction. Hence $Z(A)\subseteq \phi(L)$.
\par

Suppose now that $A$ is a quasi-minimal ideal of $L$. Then $Z(A)\subseteq A\cap \phi(L) \subseteq A$, so $A\cap \phi(L)=A$ or $Z(A)$. The former implies that $A \subseteq \phi(L)$, which is impossible since $\phi(L)$ is nilpotent. Hence $A\cap \phi(L)=Z(A)$.
\end{proof}

\section{The quasi-minimal radical}
Here we construct a radical by adjoining the quasi-minimal ideals of $L$ to its nilradical $N$.

\begin{lemma}\label{l:mchar} Quasi-minimal ideals of $L$ are characteristic in $L$.
\end{lemma}
\begin{proof} This follows from Lemma \ref{l:char}.
\end{proof} 

\begin{lemma}\label{l:irred} Let $A/Z(A)$ be a minimal ideal of $L/Z(A)$. Then $A = A^2 + Z(A)$ and $A^2$ is quasi-minimal in $L$. 
\end{lemma}
\begin{proof} Let $P = A^2$ and $\overline{L} = L/Z(A)$. Then $\overline{P}$ is an ideal of $\overline{L}$ and $\overline{A}$ is minimal, so $\overline{P} = \overline{0}$ or $\overline{A}$. The former implies that $\overline{A}$ is abelian, a contradiction. Hence $\overline{P} = \overline{A}$, so $A = P+Z(A)=A^2+Z(A)$. Also, $P=A^2=P^2$ and $[Z(P),A]=[Z(P),P]+[Z(P),Z(A)]=0$, so $P\cap Z(A)=Z(P)$. Thus $P/Z(P)=P/P \cap Z(A) \cong P+Z(A)/Z(A)=A/Z(A)$ is a minimal ideal of $L/Z(P)$. 
\end{proof}

\begin{propo}\label{p:sub} Let $A$ be quasi-minimal in $L$ and $B$ be an ideal of $L$. Then either $A \subseteq B$ or $A \subseteq C_L(B)$.
\end{propo}
\begin{proof} Clearly $A \cap B +Z(A)/Z(A)$ is an ideal of $L/Z(A)$ contained in $A/Z(A)$, so $A \cap B +Z(A)=A$ or $A \cap B+Z(A)=Z(A)$. The former implies that $A=A^2 \subseteq A \cap B \subseteq A$, whence $A=A \cap B$ and $A \subseteq B$. The latter yields that $A \cap B \subseteq Z(A)$, giving $[A,B]=[A^2,B] \subseteq [A,[A,B]] \subseteq [A,A \cap B] \subseteq [A,Z(A)]=0$ and so $A \subseteq C_L(B)$.
\end{proof}
\bigskip

The {\em quasi-minimal components} of $L$ are its quasi-minimal ideals. Write MComp($L$) for the set of quasi-minimal components of $L$, and let $E^{\dagger}(L)$ be the subalgebra generated by them. Then $E^{\dagger}(L)$ is a characteristic ideal of $L$, by Lemma \ref{l:char}.

\begin{coro}\label{c:cent} $E^{\dagger}(L) \subseteq C_L(R)$.
\end{coro}
\begin{proof} Let $A \in$ MComp$(L)$ and put $B=R$ in Proposition \ref{p:sub}. Then either $A \subseteq R$ or $A \subseteq C_L(R)$. But the former is impossible, since $A^2=A$, whence $A \subseteq C_L(R)$.
\end{proof}

\begin{coro}\label{c:comp} Distinct quasi-minimal components of $L$ commute, so \[ E^{\dagger}(L) = \sum_{P\in MComp(L)}P,\]
where $[P,Q]=0$ and $P\cap Q \subseteq Z(R)$ for all $P,Q \in$  MComp$(L)$.
\end{coro}
\begin{proof} This first assertion follows directly from Proposition \ref{p:sub}. But then $P \cap Q \subseteq Z(P) \cap Z(Q) \subseteq N$ and $[P,R]=[Q,R]=0$, using Corollary \ref{c:cent}. Hence $P \cap Q \subseteq Z(R)$.
\end{proof}

\begin{lemma}\label{l:sub} If $B$ is an ideal of $L$, then MComp($B) \subseteq $ MComp$(L) \cap B$. Moreover, if $B$ is regular, then this is an equality.
\end{lemma}
\begin{proof} Let $A$ be a quasi-minimal ideal of $B$. Then $A$ is a quasi-minimal ideal of $L$, by Lemma \ref{l:mchar}. Thus  MComp($B) \subseteq$ MComp$(L) \cap B$.
\par

Now suppose that $B$ is regular, and let $A \in $ MComp$(L) \cap B$, so $A$ is a quasi-minimal ideal of $L$ and $A \subseteq B \cap C_L(N)$, by Corollary \ref{c:cent}. Let $C/Z(A)$ be a minimal ideal of $B/Z(A)$ with $C \subseteq A$. Then $C^2 \subseteq Z(A)$ or   $C^2+Z(A)=C$. The former implies that $C^3=0$, and hence that $C$ is a nilpotent ideal of $B$.  If $B$ is nilregular, it follows from Theorem \ref{t:nil} that $C \subseteq N$, whence $[C,A]=0$ and $C \subseteq Z(A)$, a contradiction. Similarly, if $B$ is solregular, then $C\subseteq R(B) \subseteq R(L)$, by Theorem \ref{t:solv}. But then $[C,A]=0$, by Corollary \ref{c:cent}, since $A \in E^{\dagger}(L)$, leading to the same contradiction. Hence $C^2+Z(A)=C$. But now
\[ [L,C]=[L,C^2+Z(A)] \subseteq [[L,C],C]+Z(A) \subseteq [B,C]+Z(A) \subseteq C, 
\]
so $C$ is an ideal of $L$. But $A/Z(A)$ is a minimal ideal of $L/Z(A)$, so $C=Z(A)$ or $C=A$. It follows that $A/Z(A)$ is a minimal ideal of $B/Z(A)$ and $A^2=A$. Thus $A \in$ MComp$(B)$.
\end{proof}

\begin{ex}\label{e:mcomp} Note that if $B$ is not regular then the inclusion in Lemma \ref{l:sub} can be strict. For, let $L$ be as in Example \ref{e:pasha}. Then ${\mathcal O}_m$ has a unique maximal ideal ${\mathcal O}_m^+$ and  $A^+=sl(2)\otimes {\mathcal O}_m^+$ is the unique maximal ideal of $A$ (and is nilpotent). Hence MComp$(A) \subseteq A^+ \neq A$, whereas MComp$(L)=A$.
\end{ex}

\begin{propo}\label{p:soc} Let $L$ be a Lie algebra in which $C_L(N)$ is regular. Put $Z=Z(N)$, $\overline{L}=L/Z$, $\overline{S}=Soc(\overline{C_L(N)})$. Then $E^{\dagger}(L)=S^2$ and $S=E^{\dagger}(L)+Z$.
\end{propo}
\begin{proof} Let $H=C_L(N)$. Then $R(\overline{H})=0$, by Theorem \ref{t:rad}. Hence each minimal ideal of $\overline{H}$ is quasi-minimal in $\overline{H}$, and so is a quasi-minimal component of $\overline{H}$. Thus $\overline{S}\subseteq E^{\dagger}(\overline{H})$. Let $\overline{K} \in$ MComp$(\overline{H}) \subseteq$ MComp$(\overline{L})$, by Lemma \ref{l:sub}. Hence $K/Z(K)$ is a quasi-minimal ideal of $L/Z(K)$, by Lemma \ref{l:centre}. Then $K=K^2+Z$ with $K^2$ quasi-minimal in $L$, since $Z(K)=Z$ by Lemma \ref{l:irred}. Hence $K^2\in$ MComp$(L)$, so $S\subseteq E^{\dagger}(L)+Z$. 
\par

Let $P\in$ MComp$(L)$. Then $P \subseteq H$ since $E^{\dagger}(L) \subseteq H$, by Corollary \ref{c:cent}. Hence $P \in$ MComp$(L) \cap H=$ MComp$(H)$, by Lemma \ref{l:sub}. Hence $\overline{P}$ is a minimal ideal of $\overline{H}$, so $P\subseteq S$. Thus $S=E^{\dagger}(L)+Z$ and $E^{\dagger}(L)=S^2$.
\end{proof}
\bigskip

We define the {\em quasi-minimal radical} of $L$ to be $N^{\dagger}(L) = N+E^{\dagger}(L)$. From now on we will denote $N^{\dagger}(L)$ simply by $N^{\dagger}$. Then this has the property we are seeking.

\begin{theor}\label{t:gennil} If $L$ is a Lie algebra, over any field $F$, with nilradical $N$, then $C_L(N^{\dagger})=Z(N)$. In particular, $C_L(N^{\dagger}) \subseteq N^{\dagger}$.
\end{theor}
\begin{proof} Let $C = C_L(N^{\dagger})$. Then $Z(N) \subseteq C$, by Corollary \ref{c:cent}. Suppose that $Z(N) \neq C$ and let $A/Z(N)$ be a minimal ideal of $L/Z(N)$ with $A \subseteq C$. Then $[A,Z(N)]\subseteq [C,N^{\dagger}]=0$, so $Z(N)\subseteq Z(A)$. Thus $A=Z(A)$ or $Z(A)=Z(N)$. The former implies that $A\subseteq N$. But $[A,N]\subseteq [C,N^{\dagger}]=0$, so $A\subseteq Z(N)$, a contradiction. The latter implies that $A^2\subseteq E^{\dagger}\subseteq N^{\dagger}$, by Lemma \ref{l:irred}. Hence $A^3\subseteq [C,N^{\dagger}]=0$, so $A\subseteq N$, which leads to the same contradiction as before. The result follows. 
\end{proof}

\begin{propo}\label{p:double} Let $L$ be a Lie algebra in which $N^{\dagger}$ is regular. Then $N^{\dagger}(N^{\dagger})=N^{\dagger}$.
\end{propo}
\begin{proof} Clearly $N^{\dagger}(N^{\dagger}) \subseteq N^{\dagger}$. But $ E^{\dagger}(L)\subseteq E^{\dagger}(N^{\dagger})$, by putting $B=N^{\dagger}$ in Lemma \ref{l:sub}, and, clearly, $N \subseteq N(N^{\dagger})$, giving the reverse inclusion.
\end{proof}

\begin{ex}\label{e:regular} Again, Proposition \ref{p:double} does not hold if $N^{\dagger}$ is not regular. For, let $L$ be as in Example \ref{e:pasha}. Then $N^{\dagger}=A$, but $N^{\dagger}(N^{\dagger})=A^+$.
\end{ex}

Next we investigate the behaviour of $N^{\dagger}$ with respect to factor algebras, direct sums and ideals.

\begin{propo}\label{p:quotient} Let $L$ be a Lie algebra over any field, and let $I$ be an ideal of $L$. Then 
\[ \frac{N^{\dagger}(L)+I}{I}\subseteq N^{\dagger}\left(\frac{L}{I}\right).
\]
\end{propo}
\begin{proof} Clearly $N(L)+I/I\subseteq N(L/I)$. Let $A$ be a quasi-minimal ideal of $L$, so $A/Z(A)$ is a minimal ideal of $L/Z(A)$ and $A^2=A$.  Put $C=C_L(A+I/I)$. Then $Z(A)\subseteq C\cap A\subseteq A$, so $C\cap A=A$ or $C\cap A=Z(A)$. The former implies that $A=A^2\subseteq I$, whence $A+I/I \subseteq N(L/I)$. If the latter holds, then $C=C\cap (A+I)=C\cap A +I=Z(A)+I$ and $A\cap I\subseteq A\cap C=Z(A)$, whence
\[ \frac{A+I/I}{Z(A+I/I)} \cong \frac{A+I}{C}= \frac{A+I}{Z(A)+I}\cong \frac{A}{Z(A)+A\cap I}=\frac{A}{Z(A)}
\] and
\[  \left(\frac{A+I}{I}\right)^2=\frac{A+I}{I}.
\]
Thus $A+I/I$ is a quasi-minimal ideal of $L/I$ and 
\[ \frac{E^{\dagger}(L)+I}{I}\subseteq E^{\dagger}\left( \frac{L}{I}\right). 
\] The result follows.
\end{proof}
\bigskip

The above inclusion can be strict, as we shall see later.

\begin{propo}\label{p:sum}  Let $L$ be a Lie algebra over any field, and suppose that $L=I\oplus J$, where $I, J$ are ideals of $L$. Then $N^{\dagger}(L)=N^{\dagger}(I)\oplus N^{\dagger}(J)$.
\end{propo}
\begin{proof} It is easy to see that $N^{\dagger}(I)\oplus N^{\dagger}(J) \subseteq N^{\dagger}(L)$. Let $\pi_I$, $\pi_J$ be the projection maps onto $I, J$ respectively. Then $N(L)=\pi_I(N(L))\oplus \pi_J(N(L))$. Clearly $\pi_I(N(L))\subseteq N(I)$ and $\pi_J(N(L))\subseteq N(J)$, so $N(L)\subseteq N(I)\oplus N(J)$.
\par

Let $A$ be a quasi-minimal ideal of $L$, so $A/Z(A)$ is a minimal ideal of $L$ and $A^2=A$. Then
\[ A=A^2\subseteq [A,I\oplus J]=[A,I]\oplus [A,J]\subseteq A,
\] so $A=[A,I]\oplus [A,J]$. Since $A=A^2=[A,I]^2+[A,J]^2$, we also have that $[A,I]^2=[A,I]$ and $[A,J]^2=[A,J]$.
Now $[A,I]+Z(A)=Z(A)$ or $A$. The former implies that $[A,I] \subseteq Z(A)$, which gives that $[A,I]=[A,I]^2=0$. The latter yields that $A/Z(A)\cong [A,I]/Z(A)\cap [A,I]$. Now $Z(A)\cap [A,I]\subseteq Z([A,I])$, so $Z([A,I])=[A,I]$ or $Z(A)\cap [A,I]$. The former gives $[A,I]=[A,I]^2=0$ again, whereas the latter yields that $[A,I]/Z[A,I]$ is quasi-minimal and $[A,I]\in E^{\dagger}(I)$.
\par

Similarly $[A,J]=0$ or else $[A,J]\in E^{\dagger}(J)$. It follows that $E^{\dagger}(L)\subseteq E^{\dagger}(I)\oplus E^{\dagger}(J)$, whence the result.
\end{proof}

\begin{propo}\label{p:ideal2}  Let $L$ be a Lie algebra over any field, and let $I$ be a nilregular ideal of $L$. Then $N^{\dagger}(I)\subseteq N^{\dagger}(L)$.
\end{propo}
\begin{proof} Since $I$ is nilregular, we have that $N(I)\subseteq N(L)$, by Theorem \ref{t:nil} (i). Also, $E^{\dagger}(I)\subseteq E^{\dagger}(L)$, by Lemma \ref{l:sub}, whence the result.
\end{proof}
\bigskip

 The following result describes the ideals of $L$ contained in $E^{\dagger}$.

\begin{propo}\label{p:idealdagger} Let $A$ be an ideal of $L$ with $A \subseteq E^{\dagger}(L)$. Then $A=P_1+ \ldots + P_k+ Z(A)$, where $P_i$ is a quasi-minimal component of $L$ for $1 \leq i \leq k$.
\end{propo}
\begin{proof} Let $E^{\dagger}(L) = P_1 + \ldots + P_n$, where $P_i$ is a quasi-minimal component of $L$ for each $1 \leq i \leq n$. Then $P_i \subseteq A$ or $P_i \subseteq C_L(A)$ for each $i=1, \ldots, n$, by Proposition \ref{p:sub}. Let $P_i \subseteq A$ for $1 \leq i \leq k$ and $P_i \not \subseteq A$ for $k+1 \leq i \leq n$. Then $A \cap (P_{k+1} + \ldots + P_n) \subseteq Z(A)$, so $A= (P_1 + \ldots + P_k)+ Z(A)$.
\end{proof}
\bigskip

Finally we give two further characterisations of $N^{\dagger}$, valid over any field. Recall that $A/B$ is a chief factor of $L$ if $B$ is an ideal of $L$ and $A/B$ is a minimal ideal of $L/B$. 

\begin{theor}\label{t:cent} Let $L$ be a Lie algebra, over any field $F$, with radical $R$.. Then 
$$N^{\dagger}=\cap \{A+C_L(A/B) \mid A/B \text{ is a chief factor of } L\}.$$
\end{theor}
\begin{proof} Denote the given intersection by $I$, let $A/B$ be a chief factor of $L$ and let $P$ be a quasi-minimal component of $L$. Then $P \subseteq A$ or $P \subseteq C_L(A)$, by Proposition \ref{p:sub}. Hence $E^{\dagger} \subseteq I$. Moreover, $N \subseteq I$, by \cite[Lemma 4.3]{bg}, so $N^{\dagger} \subseteq I$.
\par

If $P$ is a quasi-minimal component of $L$ then $P/Z(P)$ is a chief factor of $L$. Also, if $C=C_L(P/Z(P))$ we have $[C,P]=[C,P^2]\subseteq [[C,P],P] \subseteq [Z(P),P]=0$, so $C=C_L(P)$ and $N \subseteq C$, by Corollary \ref{c:cent}. Hence $I \subseteq P+C_L(P/Z(P))=P+C_L(P)$. Now, if $P$, $Q$ are quasi-minimal components of $L$, then 
\[ (P+C_L(P))\cap (Q+C_L(Q))=P+Q+C_L(P)\cap C_L(Q),
\] since $P \subseteq C_L(Q)$ and $Q \subseteq C_L(P)$. It follows that $I \subseteq N^{\dagger}+C_L(E^{\dagger})$ and $I=N^{\dagger}+I\cap C_L(E^{\dagger})$.
\par

If 
\[ 0=N_0 \subset N_1 \subset \ldots \subset N_k=N
\] is part of a chief series for $L$ then $I \subseteq \cap_{i=1}^k C_L(N_i/N_{i-1})$, so $I$ acts nilpotently on $N$. Suppose that $N \subset I\cap C_L(E^{\dagger})$. Let $A/N$ be a minimal ideal of $L/N$ with $A \subseteq I\cap C_L(E^{\dagger})$. Then $A^2 \subseteq N$ or $A^2+N=A$. The former implies that $A \subseteq N$, since $A$ acts nilpotently on $N$, a contradiction. Hence $A=A^2+N \subseteq A^r+N$ for all $r \geq 1$. But now 
\[ [A,N]\subseteq [A^r+N,N] \subseteq N(\text{ad},A)^r+N^r,
\] so $[A,N]=0$, whence $A \subseteq C_L(E^{\dagger})\cap C_L(N)=C_L(N^{\dagger})=Z(N)$, by Theorem \ref{t:gennil}, a contradiction again. Thus $I\cap C_L(E^{\dagger})=N$ and $I=N^{\dagger}$.
\end{proof}
\medskip

We put
\begin{equation*}
I_{L}( A/B) =\{ x\in L \mid \text{ad}\,(x+B)|_{A/B}
=\text{ad}\,(a+B)|_{A/B}\text{ for some }a\in A\}.
\end{equation*}
The map ad$\,(x+B)|_{A/B}$ is called the inner derivation {\em induced} by $x$ on $A/B$. Then $I_L(A/B)=A+C_L(A/B)$, by \cite[Lemma 1.4 (i)]{tz}, so we have the following corollary.

\begin{coro}\label{c:der} Let $L$ be a Lie algebra over any field $F$. Then $N^{\dagger}$ is the set of all elements of $L$ which induce an inner derivation on every chief factor of $L$.
\end{coro}

\section{The generalised nilradical of $L$}
We define the {\em generalised nilradical} of $L$, $N^*(L)$, by
\[ \frac{N^*(L)}{N}= \hbox{ Soc}_{L/N}\left(\frac{N+C_L(N)}{N} \right)
\] As usual we denote $N^*(L)$ simply by $N^*$. The following result shows that this is, in fact, the same as the quasi-nilpotent radical.

\begin{theor}\label{t:centgnil} Let $L$ be a Lie algebra with nilradical $N$ over any field. Then $N^*=N^{\dagger}$.
\end{theor}
\begin{proof} Put $C=C_L(N)$. Let $A/Z(A)$ be a minimal ideal of $L/Z(A)$ for which $A^2=A$. Then $Z(A) \subseteq A\cap N$, so $A\cap N=A$ or $A\cap N=Z(A)$. the former implies that $A\subseteq N$, which is a contradiction, so the latter holds. It follows that $(A+N)/N \cong A/A\cap N=A/Z(A)$, so $(A+N)/N$ is a minimal ideal of $L/N$. Moreover, $[A,N]=[A^2,N]\subseteq [A,[A,N]\subseteq [A,Z(A)]=0$, so $A\subseteq C$ and $(A+N)/N\subseteq N^*/N$. Hence $N^{\dagger} \subseteq N^*$.
\par

Now let $A/N$ be a minimal ideal of $L/N$ with $A\subseteq N+C$. Then $A=N+A\cap C$. Now $Z(A\cap C)=Z(N)$, by Lemma \ref{l:centre}, so $A/N \cong A\cap C/N\cap C=A\cap C/Z(N)=A\cap C/Z(A\cap C)$. It follows that $A\cap C/Z(A\cap C)$ is a minimal ideal of $ L/Z(A\cap C)$. Thus $(A\cap C)^2$ is a quasi-minimal ideal of $L$, by Lemma \ref{l:irred}. Moreover, $(A\cap C)^2+Z(N)=Z(N)$ or $A\cap C$. The former implies that $(A\cap C)^2 \subseteq Z(N)$, which yields that $(A\cap C)^3=0$ and $A\cap C\subseteq N$, a contradiction. Hence $A\cap C=(A\cap C)^2+Z(N)\subseteq N^{\dagger}$, and so $A\subseteq N^{\dagger}$. This shows that $N^*\subseteq N^{\dagger}$.
\end{proof}
\bigskip

This last result together with Theorem \ref{t:gennil} gives the following.

\begin{theor} Let $L$ be a Lie algebra over any field $F$. Then $L/Z(N)$ is isomorphic to a subalgebra of Der$(N^*)$, and $N^*/N$ is a direct sum of minimal ideals of $L/N$ which are simple or irregular.
\end{theor}
\begin{proof} The isomorphism results from the map $\theta : L \rightarrow$ Der$(N^*)$ given by $\theta(x)=$ad\,$x\mid_{N^*}$. Let $A/N$ be a minimal ideal of $L/N$ with $A\subseteq A+C$. The $A=N+A\cap C$ and, as in the second paragraph of the proof of Theorem \ref{t:centgnil}, $(A\cap C)^2$ is quasi-minimal in $L$, which implies that $A/N$ cannot be abelian. It follows from Corollary \ref{c:minp} that $A/N$ is simple or irregular.
\end{proof}

\begin{propo}\label{p:gensoc}  Let $L$ be a Lie algebra with nilradical $N$ over a field $F$, and suppose that $C_L(N)$ is nilregular in $L$. Then
\[ \frac{N^*}{N}= \hbox{ Soc}\left(\frac{N+C_L(N)}{N} \right).
\]
\end{propo}
\begin{proof} Put $C=C_L(N)$, $D=N+C$. Let $A/N$ be a minimal ideal of $D/N$. Then $A^2+N=N$ or $A$. The former implies that $A^2 \subseteq N$, whence $A^3 \subseteq [N,N+C] \subseteq N^2$, and an easy induction shows that $A^{n+1} \subseteq N^n =0$ for some $n \in \N$. It follows that $A$ is a nilpotent ideal of $D$, which is an ideal of $L$, and thus that $A \subseteq N(D)=N+N(C)\subseteq N$, by Theorem \ref{t:nil}, a contradiction. Hence $A=A^2+N$ and 
\[ [L,A] = [L,A^2+N] \subseteq [[L,A],A]+[L,N] \subseteq [D,A]+N\subseteq A,
\] so $A/N$ is a minimal ideal of $L/N$ inside $D/N$.
\par

Now suppose that $B/N$ is a minimal ideal of $L/N$ inside $D/N$, and let $A/N$ be a minimal ideal of $D/N$ inside $B/N$. Then, by the argument in the paragraph above, $A/N$ is an ideal of $L/N$, and so $A=B$. The result follows.
\end{proof}

\begin{propo}\label{p:char0} 
\begin{itemize}
\item[(i)] If $C_L(N)$ is regular and $\phi(L)\cap Z(N)=0$ then $N^*(L) = N(L)\oplus S$, where $S$ is the socle of a maximal semisimple ideal of $L$.
\item[(ii)] Over a field of characteristic zero, $N^*(L)= N(L)\oplus S=N(L)+C_L(N)$, where $S$ is the biggest semisimple ideal of $L$.
\end{itemize}
\end{propo}
\begin{proof} This follows from Corollary \ref{c:cent0}.
\end{proof}

\begin{propo}\label{p:phistar} Let $L$ be a Lie algebra over a field of characteristic zero and let $I\subseteq N^*(L)$ be an ideal of $L$. Then
\[ \frac{ N^*(L)}{I}\subseteq N^*\left(\frac{L}{I}\right).
\] 
\end{propo}
\begin{proof} This is a special case of Proposition \ref{p:quotient}.
\end{proof}
\bigskip

As a result of Example \ref{e:regular} we define, for each non-negative integer $n$, $N^*_n$, inductively by
\[ N^*_0(L)=L \text{ and } N^*_n=N^*(N^*_{n-1}(L)) \text{ for } n>0.
\]
Clearly the series
\[ L=N^*_0(L)\supseteq N^*_1(L)\supseteq \ldots
\]
will terminate in an equality, so we put $N^*_{\infty}(L)$ equal to the minimal subalgebra in this series. It is easy to see that $N^*_{\infty}(N^*_{\infty}(L))=N^*_{\infty}(L)$. Then we have

\begin{propo}\label{p:starseries}  Let $n\in \N\cup \{0\}$, and let $I$, $J$ be ideals of the Lie algebra $L$ over the field $F$. Then
\begin{itemize}
\item[(i)]  if $N^*_k(I)$ is a nilregular ideal of $N^*_k(L)$ then $N^*_{k+1}(I)$ is a characteristic ideal of $N^*_k(L)$ for $k\geq 0$;
\item[(ii)]  if $I \subseteq N^*_n(L))$ is an ideal of $L$ then $ N^*_{n+1}(L)/I\subseteq N^*_{n+1}(L/I)$.
\item[(iii)] if $L=I\oplus J$, then $N^*_k(L)=N^*_k(I)\oplus N^*_k(J)$ for all $k \geq 0$.
\end{itemize}
\end{propo}
\begin{proof} \begin{itemize}
\item[(i)] This follows from Theorem \ref{t:nil} (i) and Lemma \ref{l:mchar}.
\item[(ii)] The case $n=1$ is given by Proposition \ref{p:phistar}. So suppose that the case $n=k$ holds, where $k\geq 1$, and let $I\subseteq N^*_k(L)$. Then $I\subseteq N^*_{k-1}(L)$. Hence
\begin{align}
 \frac{N^*_{k+1}(L)}{I}=\frac{N^*(N^*_k(L))}{I} & \subseteq N^*\left(\frac{N^*_k(L)}{I}\right)  \nonumber\\
 & \subseteq N^*\left(N^*_k\left(\frac{L}{I}\right)\right)=N^*_{k+1}\left(\frac{L}{I}\right). \nonumber
\end{align}
The result now follows by induction
\item[(iii)] This is a straightforward induction proof: the case $k=1$ is given by Proposition \ref{p:sum}
\end{itemize}
\end{proof}

\begin{coro}\label{c:starseries}  Let $n\in \N$, and let $I$, $J$ be ideals of the Lie algebra $L$ over the field $F$. Then
\begin{itemize}
\item[(i)] if $N^*_{\infty}(I)$ is nilregular, it is a characteristic ideal of $N^*_{\infty}(L)$;
\item[(ii)] if $I\subseteq N^*_{\infty}(L)$ is an ideal of $L$ then $N^*_{\infty}(L)/I\subseteq N^*_{\infty}(L/I)$.
\item[(iii)] if $L=I\oplus J$, then $N^*_{\infty}(L)=N^*_{\infty}(I)\oplus N^*_{\infty}(J)$.
\end{itemize}
\end{coro}

\section{The quasi-nilpotent radical}
Here we construct a radical by adjoining the quasi-simple ideals of $L$ to the nilradical $N$. Since quasi-simple ideals are quasi-minimal they are characteristic in $L$.

\begin{lemma}\label{l:nirred} Let $L/Z(L)$ be simple. Then $L = L^2 + Z(L)$ and $L^2$ is quasi-simple.
\end{lemma}
\begin{proof} Let $P = L^2$ and $\overline{L} = L/Z(L)$. Then $\overline{P}$ is an ideal of $\overline{L}$ and $\overline{L}$ is simple, so $\overline{P} = 0$ or $\overline{L}$. The former implies that $\overline{L}$ is abelian, a contradiction. Hence $\overline{P} = \overline{L}$, and so $L = P+Z(L)=L^2+Z(L)$. Also, $P=L^2=P^2$ and $P/Z(P)=P/P \cap Z(L) \cong (P+Z(L))/Z(L)=L/Z(L)$ is simple. 
\end{proof}

\begin{lemma}\label{l:sub1} Let $A$ be a quasi-simple ideal of $L$ and $B$ an ideal of $L$. Then either $A\subseteq B$ or $A \subseteq C_L(B)$.
\end{lemma}
\begin{proof} Since quasi-simple ideals are quasi-minimal the result follows from Proposition \ref{p:sub}. 
\end{proof}
\bigskip

The {\em quasi-simple components} of $L$ are its quasi-simple ideals. We will write SComp($L$) for the set of quasi-simple components of $L$, and put $\hat{E}(L) = <$SComp$(L)>$, the subalgebra generated by the quasi-simple components of $L$. Clearly SComp$(L) \subseteq$ MComp$(L)$, $\hat{E}(L) \subseteq E^{\dagger}(L)$ and $\hat{E}(L)$ is characteristic in $L$.

\begin{lemma}\label{l:sub2} If $B$ is an ideal of $L$, then SComp($B) =$ SComp$(L) \cap B$.
\end{lemma}
\begin{proof} If $A$ is a quasi-simple ideal of B, it is an ideal of $L$ since it is characteristic in $B$, and so SComp$(B)\subseteq$ SComp$(L)\cap B$. The reverse inclusion is clear.
\end{proof}

\begin{propo}\label{p:comp} Let $P \in$ SComp$(L)$ and let $B$ be an ideal of $L$. Then $P \in$ SComp$(B)$ or $[P,B]=0$.
\end{propo}
\begin{proof} Suppose that $[P,B] \neq 0$. We have that $P$ is a quasi-simple ideal of $L$, so $P \subseteq B$, by Lemma \ref{l:sub1}. Hence $P \in$ SComp$(B)$, by Lemma \ref{l:sub2}. 
\end{proof}

\begin{coro}\label{c:scomp} Distinct quasi-simple components of $L$ commute, so \[ \hat{E}(L) = \sum_{P\in SComp(L)}P,\]
where $[P,Q]=0$ and $P\cap Q \subseteq Z(R)$ for all $P,Q \in$  SComp$(L)$.
\end{coro}
\begin{proof} This follows easily as in Corollary \ref{c:comp}.
\end{proof}

\begin{theor}\label{t:equiv} Suppose that $L$ is a Lie algebra in which $E^{\dagger}(L)$ is regular, then $\hat{E}(L)=E^{\dagger}(L)$.
\end{theor}
\begin{proof} Let $P$ be a quasi-simple ideal of $L$. Then $N(P)$ and $R(P)$ are ideals of $E^{\dagger}(L)$, by Corollary \ref{c:scomp}. It follows that $P$ is a regular ideal of $L$ and the result follows from Corollary \ref{c:equiv}.
\end{proof}
\bigskip

Clearly, if $L$ is as in Example \ref{e:pasha} we have $\hat{E}(L)=0\neq A=E^{\dagger}(L)$, so Theorem \ref{t:equiv} does not hold for all Lie algebras.

\begin{coro}\label{c:ssoc} Let $L$ be a Lie algebra in which $E^{\dagger}(L)$ and $C_L(N)$ are regular. Put $Z=Z(N)$, $\overline{L}=L/Z$, $\overline{S}=Soc(\overline{C_L(N)})$. Then $\hat{E}(L)=S^2$ and $S=\hat{E}(L)+Z$.
\end{coro}
\begin{proof} This follows from Proposition \ref{p:soc} and Theorem \ref{t:equiv}.
\end{proof}
\bigskip

We define the {\em quasi-nilpotent radical} of $L$ to be $\hat{N}(L) = N+\hat{E}(L)$. From now on we will denote $\hat{N}(L)$ simply by $\hat{N}$. The following is an immediate consequence of Theorems \ref{t:gennil} and \ref{t:equiv}.

\begin{coro}\label{c:gennil2} Suppose that $L$ is a Lie algebra in which $N^{\dagger}(L)$ is regular. Then $C_L(\hat{N})=Z(N)$. In particular $C_L(\hat{N}) \subseteq \hat{N}$.
\end{coro}

Once more, Example \ref{e:pasha} shows that the above result does not hold without some restrictions. For, if $L$ is as in that example, then $\hat{N}(L)=0$ and $C_L(\hat{N}(L))=L$.

\begin{propo}\label{p:hideal} Let $L$ be a Lie algebra a field $F$, and let $B$ be a nilregular ideal of $L$. Then $\hat{N}(B) \subseteq \hat{N}$.
\end{propo}
\begin{proof} Under the given hypotheses $N(B)$ is a characteristic ideal of $B$ (see \cite{mak}), so $N(B) \subseteq N$. Moreover, $\hat{E}(B) \subseteq \hat{E}(L)$ by Lemma \ref{l:sub2}.
\end{proof}

\begin{propo}\label{p:hdouble} Let $L$ be a Lie algebra over any field. Then $\hat{N}(\hat{N})=\hat{N}$.
\end{propo}
\begin{proof} Clearly $\hat{N}(\hat{N}) \subseteq \hat{N}$. But $\hat{E}(\hat{N})=\hat{E}(L)$, by Lemma \ref{l:sub2}, and, clearly, $N \subseteq N(\hat{N})$, giving the reverse inclusion.
\end{proof}

\begin{propo}\label{p:phihat}  Let $L$ be a Lie algebra over any field, and let $I$ be an ideal of $L$. Then 
\[ \frac{\hat{N}(L)+I}{I}\subseteq \hat{N}\left(\frac{L}{I}\right).
\]
\end{propo}
\begin{proof} This follows exactly as in Proposition \ref{p:quotient}.
\end{proof}
\bigskip

\section{Another generalisation of the nilradical}
We put $\tilde{N}(L)/\phi(L)=$ Soc$(L/\phi(L))$. We write $\tilde{N}(L)$ simply as $\tilde{N}$. Then we see that this radical also has our desired property.

\begin{theor}\label{t:centhat} Let $L$ be a Lie algebra over any field, with nilradical $N$. Then $C_L(\tilde{N}) \subseteq Z(N) \subseteq \tilde{N}$.
\end{theor}
\begin{proof} Put $C=C_L(\tilde{N})$. Suppose first that $\phi(L)=0$. Then $L=N \dot{+} U$ where $N=$ Asoc$L$ and $U$ is a subalgebra of $L$, by \cite[Theorems 7.3 and 7.4]{frat}. Then $C=N\dot{+}C \cap U$ and $C \cap U$ is an ideal of $L$. Suppose that $C \cap U \neq 0$ and let $A$ be a minimal ideal of $L$ with $A \subseteq C \cap U$. Then $A \subseteq \tilde{N}$, so $A^2\subseteq[\tilde{N},C]=0$. Hence $A \subseteq N \cap U=0$, a contradiction. It follows that $C=N$.
\par

If $\phi(L) \neq 0$ we have 
\[ \frac{C+ \phi(L)}{\phi(L)} \subseteq C_{L/\phi(L)} \left(\frac{\tilde{N}}{\phi(L)}\right) \subseteq \frac{N}{\phi(L)}.
\] Hence $C \subseteq N$, which yields $C \subseteq Z(N)$.
\end{proof}

\begin{theor}\label{t:phifree} Let $L$ be a $\phi$-free Lie algebra over any field $F$ and suppose that $\tilde{N}(L)$ is nilregular. Then $L/C_L(\tilde{N}(L))$ is isomorphic to a subalgebra of
\[  \mathcal{M}_r^- \oplus \left(\bigoplus_{i=1}^s \text{Der}\, (A_i)\right)
\] where $\mathcal{M}_r$ is the set of $r \times r$ matrices over $F$, $r$ is the dimension of the nilradical, and $A_1, \ldots, A_s$ are the simple minimal ideals of $L$.
\end{theor}
\begin{proof} Since $L$ is $\phi$-free we have that $\tilde{N}(L)=N(L)\oplus (\oplus_{i=1}^r A_i)$ where $A_1, \ldots, A_r$ are the non-abelian minimal ideals of $L$. Also, each $A_i$ is nilregular and hence simple, by Corollary \ref{c:minp}. The map $\theta : L \rightarrow$ Der\,$(\tilde{N}(L))$ given by $\theta(x)=$ ad\,$x\mid_{\tilde{N}(L)}$ is a homomorphism with kernel $C_L(\tilde{N}(L))$. But $N(L)$ is characteristic, since it is nilregular, and the $A_i$'s are characteristic, since they are perfect, so 
\[  \text{Der}\,(\tilde{N}(L))=\text{Der}\,(N(L))\oplus \left(\bigoplus_{i=1}^s \text{Der}\, (A_i)\right),
\] whence the result.
\end{proof}

\begin{propo}\label{p:equal} $N^* \subseteq \tilde{N}$. 
\end{propo}
\begin{proof} There is a subalgebra $U/\phi(L)$ of $L/\phi(L)$ such that $L/\phi(L)=N/\phi(L)\dot{+}U/\phi(L)$, by \cite[Theorems 7.3 and 7.4]{frat}. Let $A/N$ be a minimal ideal of $L/N$ with $A \subseteq N+C_L(N)$. Then $A=N\dot{+}A \cap U$, so $[N,A]=[N,N+A\cap C]\subseteq \phi(L)$ and $A \cap U/\phi(L)$ is a minimal ideal of $L/\phi(L)$. Moreover, $N/\phi(L) \subseteq $ Soc$(L/\phi(L))$, by \cite[Theorem 7.4]{frat}. Hence $A/N \subseteq$ Soc$(L/\phi(L))$, and so $N^* \subseteq \tilde{N}$.
\end{proof}
\bigskip

In general we can have $N^* \subset \tilde{N}$ and $\tilde{N}(\tilde{N})\subset \tilde{N}$, as we will show below. Recall that the category $\mathcal{O}$ is a mathematical object in the representation theory of semisimple Lie algebras. It is a category whose objects are certain representations of a semisimple Lie algebra and morphisms are homomorphisms of representations. The formal definition and its properties can be found in \cite{humph}. As in other artinian module categories, it follows from the existence of enough projectives that each $M \in \mathcal {O}$ has a {\em projective cover} $\pi: P \rightarrow M$. Here $\pi$ is an epimorphism and is essential, meaning that no proper submodule of the projective module $P$ is mapped onto $M$. Up to isomorphism the module $P$ is the unique projective having this property (see \cite[page 62]{humph}).

\begin{ex}\label{e:proj} So let $S$  be a finite-dimensional simple Lie algebra over a field $F$ of prime characteristic, let $P$ be the projective cover for the trivial irreducible $S$-module and let $R$ be the radical of $P$. Then $R$ is a faithful irreducible $S$-module and $P/R$ is the trivial irreducible $S$-module. Let $T = P\rtimes S$ be the semidirect sum of $P$ and $S$. Then $T^2=R\rtimes S$ is a primitive Lie algebra of type $1$ and $\dim(T/T^2)=1$, say $T=T^2+Fx$. Put $L=T+Fy$ where $[x,y]=y$ and $[T^2,y]=0$.
\par

Then $\phi(T) \subseteq T^2$, so $\phi(T)$ is an ideal of $L$ and $\phi(T) \subseteq \phi(L)$, by \cite[Lemma 4.1]{frat}. But $\phi(L) \subseteq T$ and, if $M$ is a maximal subalgebra of $T$ then $M+Fy$ is a maximal subalgebra of $L$, so $\phi(L)=\phi(T)=R$.  Also $Soc(L/R)=(T^2+Fy)/R$, so $\tilde{N}(L)=T^2\oplus Fy$. However, $N(L)=R\oplus Fy$ and $C_L(N(L))=N(L)$, so $N^*(L)=N(L)\neq \tilde{N}(L)$.
\par

Moreover, $\phi(\tilde{N}(L))=0$, so $\tilde{N}(\tilde{N}(L))=Soc(\tilde{N}(L))=R\oplus Fy\neq \tilde{N}(L)$.
\par

Notice that we also have $N^*(L)/\phi(L)=N(L)/R\cong Fy$, whereas \newline $N^*(L/\phi(L)) = T^2+Fy/R$. Hence the inclusions in Propositions \ref{p:quotient}, \ref{p:phistar}, \ref{p:starseries} and Corollary \ref{c:starseries} can be strict.
\end{ex}

Note that a similar example can be constructed in characteristic $p$. Let $L$ be a finite-dimensional restricted Lie algebra over a field $F$ of prime characteristic,
and let $u(L)$ denote the restricted universal enveloping algebra of L. Then every restricted $L$-module is a $u(L)$-module and vice versa, and so there is a bijection between the irreducible restricted $L$-modules and the irreducible $u(L)$-modules. In particular, as $u(L)$ is finite-dimensional, every irreducible restricted $L$-module is finite-dimensional. So, in the above example we could take $S$ to be a restricted simple Lie algebra, as the projective cover of the trivial $S$-module again exists.

\begin{propo}\label{p:factor} If $I$ is an ideal of $L$ then 
\[ \frac{\tilde{N}+I}{I} \subseteq \tilde{N} \left(\frac{L}{I}\right).
\] Moreover, if $I\subseteq \phi(L)$, then $\tilde{N}(L)/I=\tilde{N}(L/I)$.
\end{propo}
\begin{proof} Let $A/\phi(L)$ be a minimal ideal of $L/\phi(L)$. Then
\[ \frac{A+I/I}{\phi(L)+I/I}\cong \frac{A+I}{\phi(L)+I}\cong \frac{A}{A\cap (\phi(L)+I)}.
\] Now $\phi(L)\subseteq A\cap (\phi(L)+I)$, so $A\cap (\phi(L)+I)=A$ or $\phi(L)$. But $\phi(L)+I/I \subseteq \phi(L/I)$, so the former implies that $A+I/I=\phi(L/I)$ and $A+I/I\subseteq \tilde{N}(L/I)$. If the latter holds then $A\cap I\subseteq \phi(L)$. But now, $\phi(L)/A\cap I=\phi(L/A\cap I)$, by \cite[Proposition 4.3]{frat}, so 
\[ \frac{A/A\cap I}{\phi(L/A\cap I)}= \frac{A/A\cap I}{\phi(L)/A\cap I}\cong \frac{A}{\phi(L)}.
\] It follows that $A/A\cap I \subseteq \tilde{N}(L/A\cap I)$, whence $A+I/I\subseteq \tilde{N}(L/I)$.
\par

The second assertion follows from the definition of $\tilde{N}$ and the fact that $\phi(L/I)=\phi(L)/I$.
\end{proof}

\begin{propo}\label{p:phifactor3} $\tilde{N}(L)/\phi(L)= N^*(L/\phi(L))$.
\end{propo}
\begin{proof} Suppose first that $\phi(L)=0$. Then $\tilde{N}(L)$ is the socle of $L$. Now $N(L)=$ Asoc$(L)$, by \cite[Theorem 7.4]{frat}. Also, if $A$ is a minimal ideal of $L$ with $A\not \subseteq N(L)=N$, then $[A,N]\subseteq A\cap N=0$, so $A\subseteq C_L(N)$. Hence $\tilde{N}(L)\subseteq N^*(L)$.
\par

If $\phi(L)\neq 0$ the above shows that $\tilde{N}(L/\phi(L))\subseteq N^*(L/\phi(L))$. The result now follows from Propositions \ref{p:equal} and  \ref{p:factor}.
\end{proof}

\begin{propo}\label{p:tildesum} If $L=I\oplus J$, then $\tilde{N}(L)=\tilde{N}(I)\oplus \tilde{N}(J)$. 
\end{propo}
\begin{proof} We have that $N(L)=N(I)\oplus N(J)$ and $\phi(L)=\phi(I)\oplus \phi(J)$ by \cite[Theorem 4.8]{frat}. Let $A/\phi(L)$ be a minimal ideal of $L/\phi(L)$ and suppose that $A \not \subseteq N(L)$. Then $A=A^2+\phi(L)$. But $\phi(L)=\phi(I)\oplus \phi(J)$, by \cite[Theorem 4.8]{frat}, so 
\[A=A^2+\phi(I)+\phi(J)=[A,I]+\phi(I)+[A,J]+\phi(J).
\] Hence
\[ \frac{A}{\phi(L)}\cong \frac{[A,I]+\phi(I)}{\phi(I)}\oplus \frac{[A,J]+\phi(J)}{\phi(J)}.
\]
It is easy to see that the direct summands are minimal ideals of $I/\phi(I)$ and $J/\phi(J)$ respectively, so $\tilde{N}(L)\subseteq \tilde{N}(I)\oplus \tilde{N}(J)$. Also, if $A/\phi(I)$ is a minimal ideal of $I/\phi(I)$, then $A+\phi(J)/\phi(L)$ is a minimal ideal of $L/\phi(L)$, so $\tilde{N}(I)\subseteq \tilde{N}(L)$. Similarly $\tilde{N}(J))\subseteq \tilde{N}(L)$, which gives the result.
\end{proof}
\bigskip

As a result of Example \ref{e:proj} we define, for each non-negative integer $n$, $\tilde{N}_n(L)$ inductively by
\[ \tilde{N}_0(L)=L \text{ and } \tilde{N}_n(L)=\tilde{N}(\tilde{N}_{n-1}(L)) \text{ for } n>0.
\]
Clearly the series 
\[ L=\tilde{N}_0(L) \supseteq \tilde{N}_1(L) \supseteq \ldots
\]
will terminate in an equality, so we put $\tilde{N}_{\infty}(L)$ equal to the minimal subalgebra in this series. It is easy to see that $\tilde{N}_{\infty}(\tilde{N}_{\infty}(L))=\tilde{N}_{\infty}(L)$.

\begin{propo}\label{p:prop} Let $n\in \N\cup \{0\}$, and let $I$, $J$ be ideals of the Lie algebra $L$ over a field $F$. 
\begin{itemize}
\item[(i)] If $I \subseteq \phi(\tilde{N}_{n-1}(L))$ then $\tilde{N}_n(L/I)=\tilde{N}_n(L)/I$.
\item[(ii)] $N(\tilde{N}_n(L))\subseteq N(\tilde{N}_{n+1}(L))$ for each $n\geq 0$.
\item[(iii)] If $\tilde{N}_{\infty}(L)$ is nilregular, then $\phi(\tilde{N}_{n+1}(L))\subseteq \phi(\tilde{N}_n(L))$ for each $n \geq 0$.
\item[(iv)] If $\tilde{N}_{\infty}(L)$ is nilregular then $N(\tilde{N}_n(L))=N(L)$ and $\tilde{N}_n(L)$ is an ideal of $L$ for all $n\geq 0$.
\item[(v)] If $N^*(L)$ is nilregular then $N^*(L) \subseteq \tilde{N}_n(L)$ for each $n\geq 0$.
\item[(vi)] If $\tilde{N}_n(L)$ is nilregular and $ \phi(\tilde{N}_n(L))=0$ then $\tilde{N}_{n+1}(L)=N^*(L)$.
\item[(vii)]  If $N^*(L)$ is nilregular then $C_L(\tilde{N}_n(L))=Z(N(L))$.
\item[(viii)] If $F$ has characteristic zero, then $\tilde{N}_n(I)\subseteq \tilde{N}_n(L)$.
\item[(ix)] If $F$ has characteristic zero, then $\tilde{N}_n(L)+I/I\subseteq \tilde{N}_n(L/I)$.
\item[(x)] If $L=I\oplus J$ then $\tilde{N}_n(L)=\tilde{N}_n(I)\oplus \tilde{N}_n(J)$.
\end{itemize}
\end{propo}
\begin{proof} \begin{itemize}
\item[(i)] The case $n=1$ is given by Proposition \ref{p:factor}. A straightforward induction argument then yields the general case.
\item[(ii)] We have that $N(L)\subseteq \tilde{N}(L)$, by \cite[Theorem 7.4]{frat}, whence $N(L)\subseteq N(\tilde{N}(L))$. Thus $N(\tilde{N}(L))$ $\subseteq N(\tilde{N}_2(L))$, and a simple induction argument gives the general result.
\item[(iii)] Put $\tilde{N}_i=\tilde{N}_i(L)$. Then
\[ \frac{\tilde{N}_{n+1}}{\phi(\tilde{N}_n)}= \bigoplus_{i=1}^r\frac{A_i}{\phi(\tilde{N}_n)},
\] where each direct summand is a minimal ideal of $\tilde{N}_n/\phi(\tilde{N}_n)$. Now
\[ N\left(\frac{A_i}{\phi(\tilde{N}_n)}\right)\subseteq  N\left(\frac{\tilde{N}_n}{\phi(\tilde{N}_n)}\right)=\frac{N(\tilde{N}_n)}{\phi(\tilde{N}_n)}
\] and $N(\tilde{N}_n)\subseteq N(\tilde{N}_{\infty})$ by (ii), so the direct summands are nilregular, and hence are abelian or simple, by Corollary \ref{c:minp}. It follows that they are $\phi$-free, and thus, so is $\tilde{N}_{n+1}/\phi(\tilde{N}_n)$. The result follows.
\item[(iv)] Consider the first assertion: it clearly holds for $n=0$. Suppose that $\tilde{N}_{\infty}(L)$ is nilregular and that the result holds for $k\leq n$ ($n\geq 0$). Then $\tilde{N}_k(L)$ is nilregular for all $k\geq 0$, by (ii). It follows from \cite[Corollary 1]{pet} that $N(\tilde{N}_n(L))$ is a characteristic ideal of $\tilde{N}_n(L)$, and hence an ideal of $\tilde{N}_{n-1}(L)$. Thus $N(\tilde{N}_n(L))=N(\tilde{N}_{n-1}(L))$, and so $N(\tilde{N}_n(L))=N(L)$ by the inductive hypothesis, which proves the first assertion.
\par

Put $\tilde{N}_n=\tilde{N}_n(L)$, $\phi_n=\phi(\tilde{N}_n)$ and let $A/\phi_n$ be a minimal ideal of $\tilde{N}_n/\phi_n$. If $A \not \subseteq N(\tilde{N}_n)$, then $A/\phi_n$ is a perfect subideal of $L/\phi_n$ and so an ideal of $L/\phi_n$, by Lemma \ref{l:char}. The result follows.
\item[(v)] The case $n=1$ is Proposition \ref{p:equal}. So suppose that $N^*(L) \subseteq \tilde{N}_k(L)$ for some $k \geq 1$. Then 
\[N^*(L)=N^*(N^*(L))\subseteq N^*(\tilde{N}_k(L))\subseteq \tilde{N}_{k+1}(L),
\] by Propositions \ref{p:double}, \ref{p:ideal2} and \ref{p:equal}.
\item[(vi)] If  $ \phi(\tilde{N}_n(L))=0$ then 
\[ \tilde{N}_{n+1}(L) \subseteq N^*(\tilde{N}_n(L))\subseteq N^*(L)\subseteq \tilde{N}_{n+1}(L), 
\] since $\tilde{N}_n(L)$ is nilregular (and hence so is $N^*(L)$), by Propositions \ref{p:phifactor3}, \ref{p:ideal2} and (v) above.
\item[(vii)] Using (v) above we have that $C_L(\tilde{N}_n(L))\subseteq C_L(N^*(L))=Z(N)$, by Theorem \ref{t:gennil}.
\item[(viii)] We have $\phi(I)\subseteq \phi(L)$, by \cite[Corollary 3.3]{frat}, so $\tilde{N}(L/\phi(I))=\tilde{N}(L)/\phi(I)$. Now
\[ \tilde{N}(I)/\phi(I)=N^*(I/\phi(I)\subseteq N^*(L/\phi(I)) \subseteq \tilde{N}(L/\phi(I))=\tilde{N}(L)/\phi(I),
\] by Propositions \ref{p:phifactor3}, \ref{p:ideal2} and \ref{p:equal}. Hence $\tilde{N}(I)\subseteq \tilde{N}(L)$. Then a simple induction proof shows that $\tilde{N}_n(I)\subseteq \tilde{N}_n(L)$.
\item[(ix)]The case $n=1$ is given by Proposition \ref{p:factor}. Suppose it holds for some $k\geq 1$. Then
\begin{align}& \frac{\tilde{N}_{k+1}(L)+I}{I} =\frac{\tilde{N}(\tilde{N}_k(L))+I}{I}\subseteq \frac{\tilde{N}(\tilde{N}_k(L)+I)+I}{I} \nonumber \\
 & \subseteq \tilde{N}\left(\frac{\tilde{N}_k(L)+I}{I}\right)\subseteq \tilde{N}\left(\tilde{N}_k\left(\frac{L}{I}\right)\right)=\tilde{N}_{k+1}\left(\frac{L}{I}\right), \nonumber
\end{align}
by (viii) and Proposition \ref{p:factor}.
\item[(x)] The case $n=1$ is given by Proposition \ref{p:tildesum}. A straightforward induction argument then gives the general result.
\end{itemize}
\end{proof}

\begin{coro} Let $I$, $J$ be ideals of $L$. 
\begin{itemize}
\item[(i)] If $I\subseteq \phi(\tilde{N}_{\infty}(L))$ then $\tilde{N}_{\infty}(L/I)=\tilde{N}_{\infty}(L)/I$.
\item[(ii)] If $\tilde{N}_{\infty}(L)$ is nilregular the $N(\tilde{N}_{\infty}(L))=N(L)$ and $\tilde{N}_{\infty}(L)$ is an ideal of $L$.
\item[(iii)] If $N^*(L)$ is nilregular then $N^*(L) \subseteq \tilde{N}_{\infty}(L)$.
\item[(iv)] If $\tilde{N}_{\infty}(L)$ is nilregular and $\phi(\tilde{N}_{\infty}(L)=0$ then $\tilde{N}_{\infty}(L)=N^*(L)$.
\item[(v)] If $N^*(L)$ is nilregular then $C_L(\tilde{N}_{\infty}(L))=Z(N(L))$.
\item[(vi)] If $F$ has characteristic zero, then $\tilde{N}_{\infty}(I)\subseteq \tilde{N}_{\infty}(L)$.
\item[(vii)] If $F$ has characteristic zero, then $\tilde{N}_{\infty}(L)+I/I\subseteq \tilde{N}_{\infty}(L/I)$;
\item[(viii)] If $L=I\oplus J$ then $\tilde{N}_{\infty}(L)=\tilde{N}_{\infty}(I)\oplus \tilde{N}_{\infty}(J)$.
\end{itemize}
\end{coro}

If $S$ is a subalgebra of $L$ the {\em core} of $S$, $S_L$, is the biggest ideal of $L$ contained in $S$. The following is an analogue of a result for groups given by Vasil'ev et al. in \cite{vvs}.

\begin{theor}\label{t:phi} Let $L$ be a Lie algebra over any field. Then the core of the intersection of all maximal subalgebras such that $L=M+\tilde{N}(L)$ is equal to $\phi(L)$.
\end{theor}
\begin{proof} Put $P$ equal to the intersection of all maximal subalgebras such that $L=M+\tilde{N}$. Clearly $\tilde{N}\not \subseteq \phi(L)$ and $\phi(L) \subseteq P_L$. Factor out $\phi(L)$ and suppose that $P_L\neq 0$. Let $A$ be a minimal ideal of $L$ contained in $P_L$. Then $A \subseteq \tilde{N}(L)$. 
\par

Since $\phi(L)=0$ there is a maximal subalgebra of $L$ such that $A\not \subseteq M$. If $L=\tilde{N}(L)+M$ we have $A\subseteq P_L\subseteq M$, a contradiction. If not, then $A\subseteq \tilde{N}(L)\subseteq M$, a contradiction again. Hence $P_L=0$.
\par

It follows that $P_L\subseteq \phi(L)$, whence the result.

\end{proof}

\end{document}